\documentclass[12 pt]{article}

\usepackage{amsfonts, amsmath, amssymb, amsthm, amsopn, latexsym, hyperref, graphicx, mathrsfs, verbatim, enumerate, url} % Standard packages
\usepackage{subfig} %For easily making subfigures of figures.
\usepackage[all]{hypcap} %So that hyperlinks to figures go to the figure and not the caption
\usepackage{pst-all, multido} %Package to draw figures
\usepackage[margin=1in]{geometry} %Correct margins for a5 paper
\usepackage{Style} % My personal commands
\begin{document}

\renewcommand{\qedsymbol}{$\square$}

%%%%%%%%%%%Theorem Environments%%%%%%%%%%%

\newtheorem{letthm}{Theorem}
\renewcommand*{\theletthm}{\Alph{letthm}}
\newtheorem{thm}{Theorem}[section]
\newtheorem{lem}[thm]{Lemma}
\newtheorem{cor}[thm]{Corollary}
\newtheorem{prop}[thm]{Proposition}
\theoremstyle{definition}
\newtheorem{rem}[thm]{Remark}
\newtheorem*{War}{Warning}
\newtheorem{Def}[thm]{Definition}
\newtheorem{Not}[thm]{Notation}
\newtheorem{Ass}[thm]{Assumption}
\newtheorem{ex}[thm]{Example}
\newtheorem*{exs}{Examples}
\newtheorem{obs}[thm]{Observation}
\newtheorem*{Ack}{Acknowledgements}

%%%%%%%%%%%%%%%%%%%%%%%%%%%%%%%%%

%\frontmatter

\title{Measures and dynamics on Noetherian spaces}
\author{William Gignac}
\date{\today}
\maketitle

\begin{abstract} \noindent We give an explicit description of all finite Borel measures on Noetherian topological spaces $X$, and characterize them as objects dual to a space of functions on $X$. We use these results to study the asymptotic behavior of continuous dynamical systems on Noetherian spaces.
\end{abstract}

\tableofcontents

\section{Introduction} The goal of this paper is to develop a theory of measures on Noetherian topological spaces, and use it to study dynamical systems on these spaces. Such systems arise often in complex and, more recently, nonarchimedean dynamics, where one often considers the dynamics of morphisms $f\colon X\to X$ of algebraic varieties. Viewing these as continuous dynamical systems for the Zariski topology  has led to results on the asymptotic behavior of multiplicities associated to the dynamical system (see \cite{MR1741274} and \cite{MR2513537}), which have proved useful in studying equidistribution problems (see \cite{MR2032940}, \cite{Rodrigo}, and \cite{UZ}). The present article will be used in forthcoming work \cite{Me} on equidistribution problems for dynamics in higher dimensional Berkovich spaces (see \cite{MR2578470} or \cite{Jonsson} for the $1$-dimensional case.)

Recall that a topological space $X$ is Noetherian if every descending chain of closed subsets $E_1\supseteq E_2\supseteq\cdots$ is eventually constant. An equivalent and often useful definition is the following: $X$ is Noetherian if and only if every nonempty collection of closed subsets  has an element which is minimal under inclusion. A closed set $E\subseteq X$ is said to be irreducible if it cannot be written as a union $E = E_1\cup E_2$ of two proper closed sets $E_1, E_2\subsetneq E$. In a Noetherian space, every closed set $E$ can be written as a finite union $E = E_1\cup\cdots \cup E_r$ of irreducible closed sets. Moreover, if one assumes that $E_i\not\subset E_j$ for $i\neq j$, this decomposition is unique; in this case, the $E_i$ are called the irreducible components of $E$. The proofs of these facts can be found in \cite{MR0463157}. 

While many of the results of this paper will be proved for general Noetherian spaces $X$, we spend the remainder of the section summarizing them in the specific case when $X$ is a \emph{Zariski space}, defined below.

\begin{Def}\label{assume} A Noetherian space $X$ is a \emph{Zariski space} if every nonempty irreducible closed subset $E\subseteq X$ has a unique \emph{generic point}, that is, a point which is dense in $E$.
\end{Def}

Many of the Noetherian spaces encountered in practice are Zariski spaces, so this is a case of particular interest. For instance, if $X$ is the underlying topological space of a Noetherian scheme, then $X$ is a Zariski space. The results we prove in this article take on their cleanest and simplest form for Zariski spaces.

%This assumption is satisfied, for instance, if $X$ is the underlying topological space of a Noetherian scheme, and thus is a case of particular interest. For spaces which satisfy this assumption, the results we shall prove take on their cleanest and simplest form.

In the first sections of this paper, namely \S2 through \S5, we focus solely on the measure theory of Noetherian spaces. We will begin by giving two different descriptions of their finite Borel measures. The first description, derived in \S2, is as follows.

\begin{letthm}\label{thmA} Let $X$ be a Zariski space, and let $\mu$ be a finite signed Borel measure on $X$. Then $\mu$ can be written uniquely as an absolutely convergent sum $\mu = \sum_{x\in X} c_x\delta_x$, where the $c_x$ are real numbers, and where $\delta_x$ is the Dirac probability measure at $x$.
\end{letthm}

Our second description, given in \S3, characterizes finite Borel measures on $X$ as objects dual to a space $SC(X)$ of functions on $X$, analogous to the duality between Radon measures and continuous functions on compact Hausdorff spaces. The space $SC(X)$ consists of all functions $f\colon X\to \R$ of the form $f = g - h$, where $g$ and $h$ are bounded upper semicontinuous functions on $X$. It is equipped with the supremum norm $\|f\| = \sup_X |f|$. 

\begin{letthm}\label{thmB} Let $X$ be a Zariski space. Then integration induces a duality between the real vector space $\mc{M}(X)$ of finite Borel measures on $X$ and $SC(X)$. That is, $\mc{M}(X)\cong SC(X)^*$.
\end{letthm}

The isomorphism $\mc{M}(X)\cong SC(X)^*$ allows us to pull back the weak and strong topologies from $SC(X)^*$ to $\mc{M}(X)$. In \S4, we explore the compactness properties of $\mc{M}(X)$ in its weak topology. The main result of the section is the following.

\begin{letthm}\label{thmC} Let $X$ be a Zariski space. Then any weakly closed and strongly bounded subset of $\mc{M}(X)$ is both compact and sequentially compact in the weak topology.
\end{letthm}

Finally, in \S5 we complete our purely measure theoretic studies by presenting a means for passing from an arbitrary Noetherian space $X$ to a Zariski space $\hat{X}$, without drastically changing the measure theory of $X$. This will allow us to later reduce measure theoretic problems on $X$ to problems on $\hat{X}$, where measures are better behaved. 

Starting in \S6, we begin to study continuous dynamical systems $f\colon X\to X$ on  Noetherian spaces. Our first dynamical result, proved in \S6, is a classification of all $f$-ergodic Borel probability measures on $X$.

\begin{letthm}\label{thmD} Let $X$ be a Zariski space, and let $f\colon X\to X$ be continuous. Then the $f$-ergodic probability measures on $X$ are those measures of the form \[
\mu = \frac{1}{r}(\delta_{x_1} + \delta_{x_2} + \cdots + \delta_{x_r}),\] where the $x_i$ form a periodic cycle for $f$.
\end{letthm}

In \S7, we study the asymptotic behavior of forward and reverse orbits of $f$. We present two results of a similar flavor (\hyperref[thmE]{Theorem~\ref*{thmE}} and \hyperref[thmF]{Theorem~\ref*{thmF}} below) which describe this behavior. The first of these is originally due to Favre (see \cite{thesis} and \cite{MR1741274}), but we give a new proof and measure theoretic interpretation.

\begin{letthm}[Favre]\label{thmE} Let $X$ be a Zariski space, and let $f\colon X\to X$ be a continuous map. Fix a point $x\in X$. Then there exist points $y_1,\ldots, y_r\in X$ which form a periodic cycle for $f$ such that \[
\frac{1}{n}\sum_{k=0}^{n-1} f^k_*\delta_x\to \frac{1}{r}(\delta_{y_1} + \delta_{y_2} + \cdots + \delta_{y_r})\] weakly as $n\to \infty$. Moreover, the $y_i$ have an explicit description: they are the generic points of the irreducible components of the smallest closed set which contains $f^n(x)$ for all sufficiently large $n$.
\end{letthm}

\begin{letthm}\label{thmF} Let $X$ be a Zariski space, and let $f\colon X\to X$ be a continuous surjective map. Fix a point $x\in X$ and a reverse orbit of $x$, that is, a sequence $(x_{-n})_{n=0}^\infty$ of points such that $x_0 = x$ and $f(x_{-n}) = x_{-n+1}$ for all $n\geq 1$. Then there exist points $y_1,\ldots, y_r\in X$ which form a periodic cycle for $f$ such that \[
\frac{1}{n}\sum_{k=0}^{n-1} \delta_{x_{-k}}\to \frac{1}{r}(\delta_{y_1} + \delta_{y_2} + \cdots + \delta_{y_r})\] weakly as $n\to \infty$. Moreover, the $y_i$ have an explicit description:  they are the generic points of the irreducible components of $\ol{\{x_{-n} : n\geq 0\}}$. 
\end{letthm}

These two theorems show that continuous dynamical systems on Noetherian spaces have rather strong ergodic properties. Indeed, using the isomorphism $\mc{M}(X)\cong SC(X)^*$ obtained from \hyperref[thmB]{Theorem~\ref*{thmB}}, one can restate \hyperref[thmE]{Theorem~\ref*{thmE}} as follows: for any function $\tau\in SC(X)$ and any point $x\in X$, the Birkhoff time averages $n^{-1}\sum_{k=0}^{n-1} (\tau\circ f^k)(x)$ will \emph{always} converge, namely to $r^{-1}(\tau(y_1) + \cdots + \tau(y_r))$. Similarly, for \hyperref[thmF]{Theorem~\ref*{thmF}}, the reverse time averages $n^{-1}\sum_{k=0}^{n-1} \tau(x_{-k})$ converge to $r^{-1}(\tau(y_1) + \cdots + \tau(y_r))$. Finally, we complete \S7 by using these results to generalize recent work of Dinh \cite{MR2513537}.

\begin{Ack} This work was supported by the grants DMS-0602191, DMS-0901073 and DMS-1001740. I would like to thank Mattias Jonsson for his guidance and encouragement during the course of this work, and the referee for many useful comments and suggestions.
\end{Ack}

\section{Description of Borel measures}

Fix a nonempty Noetherian topological space $X$. In this section we will give an explicit description of all finite signed Borel measures on $X$. Before doing so, we fix notation which will be used throughout this article.

\begin{Not} We will denote by $\mc{M}(X)$ the vector space of finite signed Borel measures on $X$, and by $\mc{M}(X)_+$ the cone of positive Borel measures. The Borel $\sigma$-algebra of $X$ will be written $\ms{B}(X)$. We let $\ms{E} = \ms{E}(X)$ denote the collection of all nonempty irreducible closed subsets of $X$.
\end{Not}

\begin{Def} A closed set $E\subseteq X$ is said to be \emph{$\sigma$-irreducible} if it cannot be written as a countable union $E = \bigcup E_n$ of proper closed subsets $E_n\subsetneq E$. The collection of all nonempty $\sigma$-irreducible closed subsets of $X$ will be written $\ms{F} = \ms{F}(X)$.
\end{Def}

\begin{Def}\label{Type} Let $E\subseteq X$ be a nonempty closed set, and let $A\subseteq X$ be any set. We say that $A$ has \emph{type 1 intersection} with $E$ if there exist countably many closed sets $E_n\subsetneq E$ such that $E\smallsetminus \bigcup E_n\subseteq A\cap E$. We say that $A$ has \emph{type 2 intersection} with $E$ if there exist countably many closed sets $E_n\subsetneq E$ such that $A\cap E\subseteq \bigcup E_n$. 
\end{Def}

If $E$ is a nonempty closed set which is not $\sigma$-irreducible, then any set $A$ tautologically has both type 1 and type 2 intersection with $E$. Thus \hyperref[Type]{Definition~\ref*{Type}} only has content for $\sigma$-irreducible sets $E$. On the other hand, if $E$ is $\sigma$-irreducible, it is impossible for $A$ to have both type 1 and type 2 intersection with $E$ (it may have neither). Intuitively, one should think of type 1 intersections as being ``thick," with $A\cap E$ containing ``most" of $E$. Similarly, type 2 intersections are ``thin," with $A\cap E$ containing ``hardly any" of $E$. The following easy observation will be used a number of times in this section.

\begin{obs}\label{comp} Let $E\in \ms{F}$. Then a set $A\subseteq X$ has type 1 intersection with $E$ if and only if its complement $A^c = X\smallsetminus A$ has type 2 intersection with $E$.
\end{obs}

\begin{prop}\label{borel} A set $A\subseteq X$ is a Borel set if and only if it has either type 1 or type 2 intersection with every $E\in \ms{F}$.
\end{prop}
\begin{proof} Let $\ms{A}$ be the collection of all sets $A\subseteq X$ that have either type 1 or type 2 intersection with every set $E\in \ms{F}$. Note that $\ms{A}$ contains all closed subsets of $X$. We start by showing that $\ms{A}$ is a $\sigma$-algebra, and hence that $\ms{B}(X)\subseteq \ms{A}$. \hyperref[comp]{Observation~\ref*{comp}} immediately gives that $\ms{A}$ is closed under complements, so one only needs to check $\ms{A}$ is closed under countable unions. Let $A_1,A_2,\ldots \in \ms{A}$, and let $A = \bigcup A_n$. Fix $E\in \ms{F}$. It is easy to see that if one of the $A_n$ has type 1 intersection with $E$, then so does $A$. On the other hand, if all $A_n$ have type 2 intersection with $E$, then so does $A$. Thus $A$ has either type 1 or type 2 intersection with $E$, and we conclude that $A\in \ms{A}$. Therefore $\ms{B}(X)\subseteq\ms{A}$.

It remains to show that $\ms{A}\subseteq\ms{B}(X)$. Suppose for contradiction that some $A\in \ms{A}$ is not a Borel set. Let $T$ be the collection of all closed sets $E\subseteq X$ for which $A\cap E\notin \ms{B}(E)$. Since $X\in T$, the collection $T$ is nonempty. By Noetherianity, there is a minimal element $E$ of $T$. Replacing $A$ with $A^c$ if necessary, we may assume with no loss of generality that $A$ has type 2 intersection with $E$, i.e., $A\cap E\subseteq\bigcup E_n$ for closed sets $E_n\subsetneq E$. But the minimality of $E$ implies that $A\cap E_n\in \ms{B}(E_n)\subset \ms{B}(E)$, and hence $A\cap E = \bigcup (A\cap E_n)\in \ms{B}(E)$, a contradiction.
\end{proof}

\begin{Def} Let $E\in \ms{F}$ be a $\sigma$-irreducible closed set. The \emph{Dirac mass} at $E$ is the Borel measure $\delta_E$ on $X$ defined by  \[
\delta_E(A) := \begin{cases}
1 & \mbox{if $A$ has type 1 intersection with $E$.}\\
0 & \mbox{if $A$ has type 2 intersection with $E$.}
\end{cases}\]
\end{Def}

\begin{lem}\label{lemagree} Let $\mu,\nu\in \mc{M}(X)_+$ be measures that agree on $\ms{F}$. Then $\mu = \nu$.
\end{lem}
\begin{proof} Suppose for contradiction that there is a Borel set $A\in \ms{B}(X)$ with $\mu(A)\neq \nu(A)$. Let $T$ be the collection of all closed sets $E$ such that $\mu(A\cap E)\neq \nu(A\cap E)$. Note that $T$ is nonempty, since $X\in T$. There is then a minimal element $E\in T$. We first note that $E$ must be irreducible. Indeed, if $E$ were reducible, say $E = E_1\cup E_2$, then the minimality of $E$ implies \begin{align*}
\mu(A\cap E) & = \mu(A\cap E_1) + \mu(A\cap E_2) - \mu(A\cap E_1\cap E_2)\\
& = \nu(A\cap E_1) + \nu(A\cap E_2) - \nu(A\cap E_1\cap E_2) = \nu(A\cap E),
\end{align*} a contradiction. Moreover, $E$ must be $\sigma$-irreducible, since if $E = \bigcup E_n$ for some closed sets $E_n\subsetneq E$, it would follow from the minimality of $E$ that \[\tag{$*$}
\mu(A\cap E) = \lim_{N\to \infty} \mu\left(A\cap \bigcup_{n=1}^N E_n\right) = \lim_{N\to \infty}\nu\left(A\cap \bigcup_{n=1}^N E_n\right) = \nu(A\cap E).\] Thus $E\in \ms{F}$. Since $\mu$ and $\nu$ agree on $\ms{F}$, this implies that $\mu(A^c\cap E)\neq \nu(A^c\cap E)$, and moreover that $E$ is minimal among closed sets with the property. Therefore, replacing $A$ with $A^c$ if necessary, we may assume $A$ has type 2 intersection with $E$, i.e., $A\cap E\subseteq\bigcup E_n$ for some closed sets $E_n\subsetneq E$. But then equation ($*$) applies again, so that $\mu(A\cap E) = \nu(A\cap E)$, a contradiction. This completes the proof.
\end{proof}

We are now in a position to give the main result of the section, a description of all finite Borel measures on $X$. 

\begin{thm} \label{classification}Every finite signed Borel measure $\mu\in \mc{M}(X)$ can be written uniquely as an absolutely convergent sum $\mu = \sum_{E\in \ms{F}} c_E\delta_E$, where $c_E\in \R$ for all $E\in \ms{F}$.
\end{thm}
\begin{proof} Any measure $\mu\in \mc{M}(X)$ can be written as a difference $\mu = \mu^+ - \mu^-$ of positive measures $\mu^{\pm}\in \mc{M}(X)_+$,  so it suffices to prove the theorem for positive measures. We prove the following: any $\mu\in \mc{M}(X)_+$ can be written uniquely as a convergent sum $\sum_{E\in \ms{F}} c_E\delta_E$, where $c_E\geq 0$ for each $E\in \ms{F}$.

First suppose that $\mu$ can be written as a convergent sum $\mu = \sum_{E\in \ms{F}} c_E\delta_E$, with $c_E\geq 0$ for each $E\in \ms{F}$. Since $\mu(X) = \sum_E c_E<\infty$, it follows that at most countably many of the $c_E$ are nonzero. From this one sees that $c_E = \inf\{\mu(E\smallsetminus F) : F\subsetneq E \mbox{ closed}\}$ for all $E\in \ms{F}$. Thus the coefficients $c_E$ are completely determined by $\mu$, proving the uniqueness statement. To prove existence, we start by \emph{defining} $c_E:=\inf\{\mu(E\smallsetminus F) : F\subsetneq E \mbox{ closed}\}$ for each $E\in \ms{F}$. Suppose that $E_1,\ldots, E_n\in \ms{F}$ are distinct. Reindexing if necessary, assume that $E_i\not\subset E_j$ for any $j<i$. Since the $E_i$ are irreducible, this implies $E_i\not\subset E_1\cup\cdots\cup E_{i-1}$ for each $i$. Then \begin{align*}
\infty> \mu(X) \geq \mu(E_1\cup\cdots\cup E_n) &= \mu(E_1) + \mu(E_2\smallsetminus E_1) + \cdots + \mu(E_n\smallsetminus [E_1\cup\cdots\cup E_{n-1}])\\
&\geq c_{E_1} + \cdots + c_{E_n}.
\end{align*} It follows that $\sum_{E\in \ms{F}} c_E\leq \mu(X)$. We may then define $\nu\in \mc{M}(X)_+$ by $\nu = \sum_{E\in \ms{F}} c_E\delta_E$. To complete the proof, we must show that $\mu = \nu$. Suppose for contradiction that $\mu \neq \nu$. By \hyperref[lemagree]{Lemma~\ref*{lemagree}}, the collection $T$ of all closed sets $E$ for which $\mu(E)\neq \nu(E)$ is nonempty. Let $E$ be a minimal element of $T$. If $E$ were reducible, say $E = E_1\cup E_2$, then the minimality of $E$ would imply that \[
\mu(E) = \mu(E_1) + \mu(E_2) - \mu(E_1\cap E_2) = \nu(E_1) + \nu(E_2) - \nu(E_1\cap E_2) = \nu(E),\] a contradiction. Thus $E$ is irreducible. If $E$ were not $\sigma$-irreducible, say $E = \bigcup E_n$ for closed sets $E_n\subsetneq E$, then the minimality of $E$ would imply that \[
\mu(E) = \lim_{N\to \infty}\mu\left(\bigcup_{n=1}^N E_n\right) = \lim_{N\to \infty}\nu\left(\bigcup_{n=1}^N E_n\right) = \nu(E),\] another contradiction. Thus $E$ is $\sigma$-irreducible, and we may choose a sequence of closed sets $F_n\subsetneq E$ such that $\mu(E\smallsetminus F_n)\to c_E$  and $\nu(E\smallsetminus F_n)\to c_E$. The minimality of $E$ implies that \[
\mu(E) - c_E = \lim_{n\to \infty} \mu(F_n) = \lim_{n\to \infty} \nu(F_n) = \nu(E) - c_E,\] a contradiction. Therefore $\mu = \nu$, and the proof is complete.
\end{proof}

\begin{proof}[Proof of Theorem \ref{thmA}] \phantomsection \label{pfofthmA} Assume that $X$ is a Zariski space. We will now apply \hyperref[classification]{Theorem~\ref*{classification}} to prove \hyperref[thmA]{Theorem~\ref*{thmA}}. To start, we note that if $x\in X$, then $E = \ol{\{x\}}$ is $\sigma$-irreducible. Indeed, if $F\subsetneq E$ is a proper closed subset of $E$, then $x\notin F$, since otherwise $\ol{\{x\}} = E\subseteq F$. Therefore $x$ cannot lie in \emph{any} union $\bigcup_\alpha F_\alpha$ of proper closed subsets $F_\alpha\subsetneq E$, and, in particular, $E$ cannot be covered by proper closed subsets. Since every $E\in \ms{E}$ has a generic point $x$, we conclude that $\ms{E}\subseteq \ms{F}$, and hence $\ms{E} = \ms{F}$. Moreover, the fact that every $E\in \ms{E}$ has a \emph{unique} generic point $x$ establishes a bijection between closed sets $E\in \ms{E} = \ms{F}$ and points $x\in X$. Finally, similar reasoning establishes that a Borel set $A\in \ms{B}(X)$ will have type $1$ intersection with $E\in \ms{F}$ if and only if $A$ contains the generic point $x$ of $E$. Thus $\delta_E = \delta_x$, the usual Dirac probability measure at $x$. \hyperref[thmA]{Theorem~\ref*{thmA}} is now immediate from \hyperref[classification]{Theorem~\ref*{classification}}. %To start, note that for any point $x\in X$, the set $\ol{\{x\}}$ is $\sigma$-irreducible. Thus for Zariski spaces there is a bijective correspondence between points and $\sigma$-irreducible closed sets. Moreover, we observe the following: if $x$ is the generic point of $E = \ol{\{x\}}$, then $\delta_E = \delta_x$. Indeed, a set $A\in \ms{B}(X)$ has type 1 intersection with $E$ if and only if it contains $x$. With these facts in mind, \hyperref[thmA]{Theorem~\ref*{thmA}} is a direct translation of \hyperref[classification]{Theorem~\ref*{classification}}.
\end{proof}

%Assume now that $X$ satisfies \hyperref[assume]{Assumption~\ref*{assume}}. Then any $\sigma$-irreducible closed subset of $X$ has a unique generic point. Conversely, if $x\in X$, then $E = \ol{\{x\}}$ is $\sigma$-irreducible. Thus one has a bijective correspondence between points of $X$ and $\sigma$-irreducible closed sets. To derive \hyperref[thmA]{Theorem~\ref*{thmA}} from \hyperref[classification]{Theorem~\ref*{classification}}, we observe the following: if  $x$ is the generic point of $E\in \ms{F}$, then $\delta_E = \delta_x$. Indeed, $A\in \ms{B}(X)$ has type 1 intersection with $E$ if and only if $x\in A$. This proves \hyperref[thmA]{Theorem~\ref*{thmA}}.

\begin{ex} We now consider an example which illustrates the importance of the set $\ms{F}$ in this description of measures. Let $X$ be a countable set, and $Y$ an uncountable set. Equip both with the cofinite topology. Then $X$ and $Y$ are both Noetherian spaces. Since the only $\sigma$-irreducible closed subsets of $X$ are points, all measures on $X$ are countable sums of Dirac masses at points. Because $Y$ is uncountable, however, it is itself $\sigma$-irreducible. Thus we obtain an additional measure $\delta_Y$ on $Y$. The Borel sets on $Y$ are the countable and cocountable subsets, and one has $\delta_Y(A) = 1$ if $A$ is uncountable and $\delta_Y(A) = 0$ if $A$ is countable.
\end{ex}

\section{Measures and semicontinuous functions}

As always, let $X$ be a nonempty Noetherian topological space. Recall that a map $f\colon X\to \R$ is said to be upper semicontinuous if $\{x\in X : f(x)\geq r\}$ is closed for all $r\in \R$. We will denote by $SC(X)$ the vector space of functions $f\colon X\to \R$ of the form $f = g - h$, where $g$ and $h$ are bounded upper semicontinuous functions on $X$. This is the smallest vector space of functions on $X$ containing the bounded upper semicontinuous functions. We equip $SC(X)$ with the supremum norm $\|f\| = \sup_X|f|$. The goal of this section is to show that for certain Noetherian spaces $X$, there is a duality $\mc{M}(X)\cong SC(X)^*$. Here $SC(X)^*$ denotes the continuous dual of $SC(X)$. This duality is analogous to the duality between Radon measures and continuous functions on compact Hausdorff spaces.

\begin{Def} A functional $\varphi\in SC(X)^*$ is \emph{positive} if $\varphi(f)\geq 0$ whenever $f\geq 0$. The set of all positive functionals forms a cone in $SC(X)^*$ which we will denote $SC(X)^*_+$.
\end{Def}

\begin{prop}\label{riesz} Every $\varphi\in SC(X)^*$ can be written as $\varphi = \varphi^+ - \varphi^-$ where $\varphi^{\pm}\in SC(X)^*_+$ and $\|\varphi^{\pm}\|\leq \|\varphi\|$.
\end{prop}
\begin{proof} The proof will use the basic theory of Riesz spaces. For a reference see chapter II of \cite{MR2018901}. We begin by showing that $SC(X)$ is a Riesz space. Let $f\in SC(X)$, say $f = g - h$, where $g$ and $h$ are bounded upper semicontinuous functions. Then \[
\max(f,0) = g - \min(g,h).\] Since $\min(g,h)$ is upper semicontinuous, we see that $\max(f,0)\in SC(X)$. If $f_1,f_2\in SC(X)$, then $\max(f_1,f_2) = f_1 + \max(f_2 - f_1, 0)$ and $\min(f_1,f_2) = f_1 - \max(f_1 - f_2,0)$, and hence $\max(f_1,f_2)\in SC(X)$ and $\min(f_1,f_2)\in SC(X)$. Therefore $SC(X)$ is a Riesz space. The proposition then follows from Theorem II.2.1 of \cite{MR2018901}.
\end{proof}

\begin{prop} The linear map $\Lambda\colon \mc{M}(X)\to SC(X)^*$ given by integration \[
\Lambda(\mu)(f):= \int_X f\,d\mu\] is injective. Moreover, $\Lambda$ maps the cone $\mc{M}(X)_+$ into the cone $SC(X)^*_+$.
\end{prop}
\begin{proof} For any closed set $E\subseteq X$, the characteristic function $\chi_E$ of $E$ is upper semicontinuous. Since $\Lambda(\mu)(\chi_E) = \mu(E)$ for any closed set $E$, \hyperref[lemagree]{Lemma~\ref*{lemagree}} implies that $\Lambda$ is injective. The fact that $\Lambda$ maps positive measures to positive functionals is obvious.
\end{proof}

\begin{lem}\label{lemagree2} Let $\mc{B}$ denote the collection of characteristic functions $\chi_E$ of nonempty irreducible closed sets $E$. Then $\mc{B}$ is a linearly independent family which spans a dense subspace of $SC(X)$. In particular, any two $\varphi,\psi\in SC(X)^*$ which agree on $\mc{B}$ must be equal.
\end{lem}
\begin{proof} We first prove linear independence. Suppose for contradiction that one has a linear dependence $c_1\chi_{E_1} + \cdots + c_r\chi_{E_r} = 0$ among elements of $\mc{B}$, with $c_i\neq 0$ for each $i$. Considering the supports of these functions, it follows that $E_i\subseteq\bigcup_{j\neq i} E_j$ for each $i$, contradicting the irreducibility of the $E_i$. Thus $\mc{B}$ is a linearly independent family.

Next we prove that for any closed (not necessarily irreducible) set $F$, the characteristic function $\chi_F$ lies in the span of $\mc{B}$. Suppose for contradiction that this is not the case, and let $T$ be the collection of closed sets $F$ for which $\chi_F$ does not lie in the span of $\mc{B}$. Choose a minimal element $F$ of $T$. Clearly $F$ cannot be irreducible, or else $\chi_F\in \mc{B}$ by definition. Let $F = F_1\cup F_2$ be a nontrivial decomposition of $F$. But then $\chi_F = \chi_{F_1} + \chi_{F_2} - \chi_{F_1\cap F_2}$ lies in the span of $\mc{B}$ by the minimality of $F$, a contradiction. Thus the span of $\mc{B}$ contains all functions $\chi_F$ for $F$ closed.

To complete the proof, we only need to show that the span of $\mc{B}' = \{\chi_F : F\mbox{ closed}\}$ is dense in $SC(X)$. Let $f\in SC(X)$, let $a = \inf_X f$, and let $b = \sup_X f$. We may suppose without loss of generality that $a \neq b$, as otherwise $f = a\chi_X$ clearly lies in the span of $\mc{B}'$. For any partition $\pi = \{a = r_0 < r_1<\cdots < r_n = b\}$ of the interval $[a,b]$, let \[
f_\pi := r_0\chi_X + \sum_{i=1}^n (r_i - r_{i-1})\chi_{\{f\geq r_i\}}.\] By construction, $f_\pi\in \mathrm{span}(\mc{B}')$ for each partition $\pi$ and $\|f - f_\pi\|\leq \mathrm{mesh}(\pi)$. Thus as we let $\mathrm{mesh}(\pi)\to 0$, the functions $f_\pi$ converge uniformly to $f$.
\end{proof}

%\begin{lem}\label{lemagree2} The set of characteristic functions $\chi_E$ of closed sets $E$ spans a dense subspace of $SC(X)$. In particular, any two $\varphi, \psi\in SC(X)^*$ which agree on the $\chi_E$ must be equal.
%\end{lem}
%\begin{proof} Let $f\colon X\to \R$ be a bounded upper semicontinuous function. We must show that $f$ is a uniform limit of finite linear combinations of the $\chi_E$. Let $a = \inf_X f$ and $b = \sup_X f$. Clearly we may suppose $a\neq b$, as otherwise $f = a\chi_X$. For any partition $\pi = \{a = r_0<r_1<\cdots<r_n = b\}$ of $[a,b]$, let \[
%f_\pi = r_0\chi_X + \sum_{i=1}^n (r_i - r_{i-1})\chi_{\{f\geq r_i\}}.\] Then $f_\pi\to f$ uniformly as $\mathrm{mesh}(\pi)\to 0$.
%\end{proof}

Until this point we have made no assumption on $X$ beyond it being Noetherian. In order to proceed, we now need to assume that $X$ has the following additional topological property.

\begin{Def} A Noetherian topological space $X$ is said to be \emph{complete} if every irreducible closed subset of $X$ is $\sigma$-irreducible, that is, if $\ms{E} = \ms{F}$.
\end{Def}

We saw in the \hyperref[pfofthmA]{proof of Theorem A} that Zariski spaces are always complete. Thus in order to prove \hyperref[thmB]{Theorem~\ref*{thmB}}, it suffices to show $\mc{M}(X)\cong SC(X)^*$ for complete spaces, or equivalently that $\Lambda$ is surjective for complete spaces. Using \hyperref[riesz]{Proposition~\ref*{riesz}}, we only must show that $\Lambda$ maps $\mc{M}(X)_+$ surjectively onto $SC(X)^*_+$. This is our next theorem.

\begin{thm}Suppose $X$ is complete. Then $\Lambda\colon\mc{M}(X)_+\to SC(X)^*_+$ is surjective.
\end{thm}
\begin{proof} Let $\varphi\in SC(X)_+^*$. For each $E\in \ms{E}$, define $c_E = \inf\{\varphi(\chi_E - \chi_F) : F\subsetneq E \mbox{ closed}\}\geq 0$. Suppose $E_1,\ldots, E_n\in \ms{E}$ are distinct. Reindexing if necessary, we may assume $E_i\not\subset E_j$ for $j<i$. Since the $E_i$ are irreducible, this implies $E_i\not\subset E_1\cup\cdots\cup E_{i-1}$ for all $i$. Then \begin{align*}
\infty>\|\varphi\|\geq \varphi(\chi_{E_1\cup\cdots\cup E_n}) &= \varphi(\chi_{E_1}) + \varphi(\chi_{E_2} - \chi_{E_1\cap E_2}) + \cdots + \varphi(\chi_{E_n} - \chi_{E_n\cap (E_1\cup\cdots\cup E_{n-1})})\\
&\geq c_{E_1} + \cdots + c_{E_n}.
\end{align*} Thus $\sum_{E\in \ms{E}} c_E \leq \|\varphi\|$. We may then define a measure $\mu = \sum_{E\in \ms{E}} c_E\delta_E\in \mc{M}(X)_+$. We will show that $\varphi = \Lambda(\mu)$. Suppose that $\varphi\neq \Lambda(\mu)$. By \hyperref[lemagree2]{Lemma~\ref*{lemagree2}}, the collection $T$ of closed sets $E$ for which $\varphi(\chi_E)\neq \Lambda(\mu)(\chi_E)$ is nonempty. Let $E$ be a minimal element of $T$. Suppose that $E$ were reducible, say $E = E_1\cup E_2$. Then by the minimality of $E$, \begin{align*}
\varphi(\chi_E) &= \varphi(\chi_{E_1}) + \varphi(\chi_{E_2}) - \varphi(\chi_{E_1\cap E_2})\\
& = \Lambda(\mu)(\chi_{E_1}) + \Lambda(\mu)(\chi_{E_2}) - \Lambda(\mu)(\chi_{E_1\cap E_2}) = \Lambda(\mu)(\chi_E),
\end{align*} a contradiction. Therefore $E$ is irreducible, and, because $X$ is complete, also $\sigma$-irreducible. We may then choose a sequence $F_n\subsetneq E$ of closed sets such that $\varphi(\chi_E - \chi_{F_n})\to c_E$ and $\Lambda(\mu)(\chi_E - \chi_{F_n}) = \mu(E\smallsetminus F_n)\to c_E$. But then by the minimality of $E$, we see that \[
\varphi(\chi_E) - c_E = \lim_{n\to \infty} \varphi(\chi_{F_n}) = \lim_{n\to \infty}\Lambda(\mu)(\chi_{F_n}) = \Lambda(\mu)(\chi_E) - c_E,\] a contradiction. Therefore $\varphi = \Lambda(\mu)$, and $\Lambda$ is surjective.
\end{proof}

\begin{ex} We now give an example showing how $\Lambda$ can fail to be surjective when $X$ is not complete. Let $X$ be a countably infinite set, equipped with the cofinite topology. It is easy to see that $SC(X)$ consists of functions $f\colon X\to \R$ of the form $f(x) = c + c_x$, where $c\in \R$ is a constant independent of $x$, and where $\{x\in X : |c_x|>\eps\}$ is finite for any $\eps>0$. Let $\varphi\in SC(X)^*$ be the positive bounded linear functional defined by $\varphi(f) = c$. We will show that $\varphi$ does not lie in the image of $\Lambda$. Indeed, suppose $\varphi = \Lambda(\mu)$ for some measure $\mu$. By \hyperref[classification]{Theorem~\ref*{classification}} one has $\mu = \sum_{x\in X} a_x\delta_x$. Then $0 = \varphi(\chi_x) = \Lambda(\mu)(\chi_x) = \mu(x) = a_x$ for each $x\in X$, so $\mu = 0$. Therefore $\varphi = \Lambda(\mu) = 0$, a clear contradiction. As we will see in \S5, $\Lambda$ will \emph{always} fail to be surjective when $X$ is not complete.
\end{ex}

\section{Compactness properties of measures}

In \S3, we saw that for complete Noetherian topological spaces $X$ there is a natural duality $\mc{M}(X)\cong SC(X)^*$. Via this isomorphism, one can pull back both the strong and the weak-$*$ topologies on $SC(X)^*$ to $\mc{M}(X)$. In order to keep up the analogy with Radon measures, we will refer to the weak-$*$ topology simply as the \emph{weak topology} on $\mc{M}(X)$. The convergence of a sequence in the weak topology has a simple characterization which follows from \hyperref[lemagree2]{Lemma~\ref*{lemagree2}}: a sequence $\mu_n$ of measures converges weakly to a measure $\mu$ if and only if $\mu_n(E)\to \mu(E)$ for every (irreducible) closed set $E$.

In this section we study compactness properties of $\mc{M}(X)$ in its weak topology, proving \hyperref[thmC]{Theorem~\ref*{thmC}}. To get us started, note that Alaoglu's theorem implies the following.

\begin{thm} Suppose that $X$ is a complete Noetherian space. Then any weakly closed and strongly bounded subset of $\mc{M}(X)$ is compact in the weak topology. In particular, the set of probability measures on $X$ is weakly compact.
\end{thm}

Unfortunately, it is not obvious that compactness can be replaced with \emph{sequential compactness} in the previous theorem. To complete the proof of \hyperref[thmC]{Theorem~\ref*{thmC}} we must therefore prove the following theorem.

\begin{thm} Suppose that $X$ is a complete Noetherian space. Then any weakly closed and strongly bounded subset of $\mc{M}(X)$ is sequentially compact in the weak topology. In particular, the set of probability measures on $X$ is weakly sequentially compact.
\end{thm}

To prove this theorem, it suffices to show that the closed unit ball of $\mc{M}(X)$ is weakly sequentially compact. Suppose we are given a sequence $\mu_n$ of measures lying in the closed unit ball of $\mc{M}(X)$. By \hyperref[riesz]{Proposition~\ref*{riesz}}, we note that one can always write any of the $\mu_n$ as a difference $\mu_n = \mu_n^+ - \mu_n^-$ of positive measures lying in the closed unit ball. Thus without loss of generality we may assume that each $\mu_n$ is positive. We will therefore prove the following.

\begin{thm}\label{seq} Let $X$ be a complete Noetherian space. Any sequence $\mu_n$ of positive measures  lying in the closed unit ball of $\mc{M}(X)$ has a weakly convergent subsequence.
\end{thm}

The proof of this theorem will be by a (fairly technical) Zorn's lemma argument. Before proceeding, we will need some preliminary definitions. We say that two infinite subsets $I$ and $J$ of $\N$ are equivalent, written $I =^*J$, if $I\cap \{n,n+1,\ldots\} = J\cap \{n,n+1,\ldots\}$ for some $n$. Similarly, we will write $J\subseteq^*I$ if $J\cap \{n,n+1,\ldots\}\subseteq I\cap \{n,n+1,\ldots\}$ for some $n$. The relation $=^*$ is an equivalence relation on the set of infinite subsets of $\N$, and $\subseteq^*$ gives a partial order on the equivalence classes. We will denote the equivalence class of an infinite subset $I\subseteq \mathbb{N}$ by $[I]$.

Let $\mu_n$ be a sequence of positive measures lying in the closed unit ball of $\mc{M}(X)$, as in the statement of \hyperref[seq]{Theorem~\ref*{seq}}. Given an equivalence class $[I]$ and any positive measure $\mu\in \mc{M}(X)$, we let $A([I], \mu)$ be the collection of all closed sets $E\subseteq X$ such that $\lim_{i\to \infty, i\in I}\mu_i(F) = \mu(F)$ for all closed sets $F\subseteq E$. \hyperref[seq]{Theorem~\ref*{seq}} will be proved if we can find a pair $([I], \mu)$ such that $X\in A([I], \mu)$. Note that if $E\in A([I],\mu)$, then $F\in A([I], \mu)$ for all closed $F\subseteq E$.

Let $\ms{S}$ be the set of all pairs $([I], \mu)$, where $I$ is an infinite subset of $\N$ and $\mu\in \mc{M}(X)$ is a positive measure with the property  that \[
\mu(X) = \sup_{E\in A([I], \mu)} \mu(E).\]  Intuitively, this condition means that $\mu$ is essentially determined on the sets $E\in A([I], \mu)$. The next lemma makes this precise.

\begin{lem}\label{Slem} Let $([I], \mu)\in \ms{S}$. Then for any Borel set $A\in \ms{B}(X)$, one has \[
\mu(A) = \sup_{E\in A([I],\mu)} \mu(A\cap E).\]
\end{lem}
\begin{proof} Since $([I], \mu)\in \ms{S}$, there is a sequence $E_n\in A([I],\mu)$ such that $\mu(X) = \lim_n \mu(E_n)$. It follows immediately that $\mu(X) = \lim_n \mu(A\cup E_n)$. Therefore \[
\mu(X) = \lim_n \mu(A\cup E_n) = \lim_n \mu(E_n) + \mu(A\smallsetminus E_n) = \mu(X) + \lim_n \mu(A\smallsetminus E_n),\] so $\lim_n \mu(A\smallsetminus E_n) = 0$. In particular, $\mu(A\cap E_n)\to \mu(A)$ as $n\to \infty$, proving the lemma.
\end{proof}

Note that $\ms{S}$ is nonempty, since $([\N], 0)\in \ms{S}$. We put a partial order $<$ on $\ms{S}$ by saying $([I], \mu)< ([J], \nu)$ if the following three statements are true: \begin{enumerate}
\item[1.] $J\subseteq^* I$.
\item[2.] $A([I], \mu)\subseteq A([J], \nu)$.
\item[3.] $\nu - \mu$ is a positive, nonzero measure.
\end{enumerate} The basis for our proof of \hyperref[seq]{Theorem~\ref*{seq}} is the next proposition.

\begin{prop} \label{technical} If $([I], \mu)$ is a maximal element of $\ms{S}$, then $X\in A([I], \mu)$, and thus the sequence $\{\mu_i\}_{i\in I}$ converges weakly to $\mu$.
\end{prop}
\begin{proof} Suppose for contradiction that $X\notin A([I], \mu)$. Let $F$ be a minimal closed set not lying in $A([I],\mu)$. Then $\{\mu_i(E)\}_{i\in I}$ converges to $\mu(E)$ for all closed sets $E\subsetneq F$ but $\{\mu_i(F)\}_{i\in I}$ does not converge to $\mu(F)$. It follows, in particular, that $F$ must be irreducible, since if $F$ were reducible with decomposition $F = F_1\cup F_2$, then \begin{align*}
\lim_{\substack{i\to \infty\\ i\in I}} \mu_i(F) & = \lim_{\substack{i\to\infty\\ i\in I}} [\mu_i(F_1) + \mu_i(F_2) - \mu_i(F_1\cap F_2)]\\
& = \mu(F_1) + \mu(F_2) - \mu(F_1\cap F_2) = \mu(F),
\end{align*} a contradiction. For any $E\in A([I], \mu)$, one has that \[
\liminf_{\substack{i\to \infty\\ i\in I}} \mu_i(F) \geq \liminf_{\substack{i\to \infty\\ i\in I}} \mu(F\cap E) = \mu(F\cap E),\] and thus by \hyperref[Slem]{Lemma~\ref*{Slem}}, \[
\liminf_{\substack{i\to \infty\\ i\in I}} \mu_i(F)\geq \sup_{E\in A([I],\mu)} \mu(F\cap E) = \mu(F).\] Since $\{\mu_i(F)\}_{i\in I}$ does not converge to $\mu(F)$, it follows that $a := \limsup_{i\to \infty, i\in I} \mu_i(F)>\mu(F)$.

Let $J\subseteq I$ be an infinite subset such that the sequence $\{\mu_j(F)\}_{j\in J}$ converges to $a$, and let $\nu = \mu + (a - \mu(F))\delta_F$. We will prove that $([J],\nu)\in \ms{S}$ and $([I],\mu)<([J], \nu)$, contradicting the fact that $([I],\mu)$ is an upper bound in $\ms{S}$. This will complete the proof. First note that if $E\in A([I],\mu)$, then $F\not\subset E$, and thus \[
\lim_{\substack{j\to \infty\\ j\in J}} \mu_j(E) = \lim_{\substack{i\to\infty\\ i\in I}} \mu_i(E) = \mu(E) = \nu(E).\] In other words, $A([I], \mu)\subseteq A([J],\nu)$. It then only remains to show that $([J],\nu)\in \ms{S}$. By construction $\{\mu_j(E)\}_{j\in J}$ converges to $\nu(E)$ for all closed sets $E\subseteq F$, and hence $F\in A([J],\nu)$. It follows that $F\cup E\in A([J],\nu)$ for any $E\in A([I],\mu)$. Thus \begin{align*}
\sup_{G\in A([J],\nu)} \nu(G)& \geq \sup_{E\in A([I],\mu)} \nu(E\cup F)  = \sup_{E\in A([I],\mu)} \nu(E) + \nu(F\smallsetminus E)\\ &\geq \sup_{E\in A([I],\mu)} \mu(E) + (a - \mu(F))
 = \mu(X) + (a - \mu(F)) = \nu(X).
\end{align*} The reverse inequality $\sup_{G\in A([J],\nu)} \nu(G)\leq \nu(X)$ is trivial, since $\nu$ is positive. We therefore have $([J],\nu)\in \ms{S}$, completing the proof.
\end{proof}

\begin{proof}[Proof of Theorem \ref{seq}] By \hyperref[technical]{Proposition~\ref*{technical}}, we must show that $\ms{S}$ has a maximal element, which we will do using Zorn's lemma. Fix some totally ordered subset $C\subseteq\ms{S}$. Let $\alpha$ be the ordinal number of the same order type as $C$, so that $C = \{([I_\beta], \nu_\beta)\}_{\beta\in \alpha}$. If $\alpha$ has a maximal element $\beta\in \alpha$, then $([I_\beta], \nu_\beta)$ is an upper bound of $C$, so we may without loss of generality assume that $\alpha$ has no maximal element, i.e., that $\alpha$ is a limit ordinal. Note that $\alpha$ must be countable. Indeed, the map $\beta\in \alpha\mapsto \nu_\beta(X)$ is an order-embedding of $\alpha$ into $\R$. Since no uncountable ordinal order-embeds into $\R$, it follows that $\alpha$ is countable. We may therefore choose an increasing sequence $\beta_k\in \alpha$ such that $\alpha = \sup_k \beta_k$. Let $\nu\in \mc{M}(X)$ be the measure defined by $\nu(A) = \lim_{k\to \infty} \nu_{\beta_k}(A)$ for all Borel sets $A\in \ms{B}(X)$, and let $J\subseteq \N$ be an infinite subset such that $J\subseteq^*I_{\beta_k}$ for each $k$ (such a $J$ can be constructed by the standard diagonalization procedure). We will show that $([J],\nu)$ lies in $\ms{S}$ and is an upper bound of $C$. Suppose that $E\in A([I_{\beta_K}], \nu_{\beta_K})$ for some $K$. Recall that that this implies $E\in A([I_{\beta_k}], \nu_{\beta_k})$ for all $k\geq K$. Thus for $k\geq K$  \[
\lim_{\substack{j\to\infty\\ j\in J}} \mu_j(E) = \lim_{\substack{i\to\infty\\ i\in I_{\beta_k}}} \mu_i(E) = \nu_{\beta_k}(E).\] In particular, $\nu_{\beta_k}(E)$ is constant for all $k\geq K$, so that $\nu_{\beta_k}(E) = \nu(E)$ for $k\geq K$. It follows that $E\in A([J], \nu)$, and hence that $\bigcup_kA([I_{\beta_k}], \nu_{\beta_k})\subseteq A([J],\nu)$. Using this, we can compute \[
\nu(X) = \sup_k \nu_{\beta_k}(X) = \sup_k\sup_{E\in A([I_{\beta_k}], \nu_{\beta_k})} \nu_{\beta_k}(E) = \sup_k\sup_{E\in A([I_{\beta_k}], \nu_{\beta_k})} \nu(E) \leq \sup_{E\in A([J],\nu)} \nu(E).\] Of course, the opposite inequality $\sup_{E\in A([J], \nu)} \nu(E)\leq \nu(X)$ is trivial since $\nu$ is a positive measure. Thus we conclude that $\nu(X) = \sup_{E\in A([J], \nu)} \nu(E)$, proving that $([J], \nu)\in \ms{S}$. In order to prove $([J], \nu)$ is an upper bound of $C$, is only remains to show is that $\nu - \nu_{\beta_k}$ is a nonzero positive measure for each $k$. By construction, $\nu - \nu_{\beta_k}$ is a positive measure for each $k$, and since $\nu_{\beta_{k+1}} - \nu_{\beta_k}$ is positive and nonzero, it follows that $\nu - \nu_{\beta_k} = (\nu - \nu_{\beta_{k+1}}) + (\nu_{\beta_{k+1}} - \nu_{\beta_k})$ is positive and nonzero. Thus $([J], \nu)$ is an upper bound of $C$, and the hypotheses of Zorn's lemma are satisfied.
\end{proof}

\section{Completions of Noetherian spaces}

In \S3, we saw that there is a duality $\mc{M}(X)\cong SC(X)^*$ whenever $X$ is a complete space, but that such a duality fails to exist when $X$ is not complete. In general, the measure theory of non-complete  spaces lacks the nice properties of the measure theory of complete Noetherian spaces, such as, for instance, the compactness properties studied in \S4. In this section, we will outline a general method for passing from an arbitrary Noetherian space to a complete (in fact Zariski) space without losing any generality. This process of \emph{completion} generalizes how one obtains a scheme from a variety (on the level of topological spaces).

\begin{Def} Let $X$ be a Noetherian topological space. The \emph{completion} of $X$ is the space $\hat{X}$ defined as follows. As a set, $\hat{X}:= \ms{E}(X)$, that is, the points of $\hat{X}$ are all nonempty irreducible closed subsets of $X$. The topology on $\hat{X}$ is given by declaring the closed sets of $\hat{X}$ to be those sets of the form $V_E:=\{F\in \hat{X} : F\subseteq E\}$, where $E\subseteq X$ is closed.
\end{Def}

It is not difficult to see that the irreducible closed subsets of $\hat{X}$ are those of the form $V_E$, where $E\subseteq X$ is an irreducible closed set. Moreover, the point $E\in \hat{X}$ is the unique generic point of $V_E$. In particular, $\hat{X}$ is a Zariski space, and hence is complete. The bijection $E\mapsto V_E$ between irreducible closed subsets of $X$ and $\hat{X}$ immediately gives the following proposition.

\begin{prop}\label{mono} There is a canonical monomorphism $j\colon \mc{M}(X)\to \mc{M}(\hat{X})$ which takes a measure $\mu = \sum_{E\in \ms{F}(X)} c_E\delta_E$ to $j(\mu) = \sum_{E\in \ms{F}(X)} c_E\delta_E$, where $\delta_E\in \mc{M}(\hat{X})$ denotes the Dirac probability measure at the point $E\in \hat{X}$. Observe that this map $j$ is surjective if and only if $X$ is complete.
\end{prop}

\begin{Def} Let $X$ be a Noetherian space, and suppose $f\colon X\to \R$ is a bounded upper semicontinuous function. Let $E\subseteq X$ be a nonempty irreducible closed set. Then the \emph{generic value} of $f$ on $E$ is defined to be $f(E) := \inf_E f$. More generally, if $f\in SC(X)$ is decomposed as $f = g - h$, where $g$ and $h$ are bounded upper semicontinuous functions, then the generic value of $f$ on $E$ is defined to be $f(E) := g(E) - h(E)$. This is independent of the choice of $g$ and $h$.
\end{Def}

\begin{prop}\label{iso} There is a canonical isometric isomorphism $\eta\colon SC(X)\to SC(\hat{X})$ which takes $f\in SC(X)$ to the function $\eta(f)$ defined by $\eta(f)(E) := f(E)$.
\end{prop}
\begin{proof} We first note that $\eta$ does in fact map $SC(X)$ to $SC(\hat{X})$. To see this, it suffices to show that $\eta$ maps bounded upper semicontinuous functions to bounded upper semicontinuous functions. Suppose $f$ is a bounded upper semicontinuous function. Then \[
\{E\in \hat{X} : \eta(f)(E)\geq r\} = \{E\in \hat{X} : {\inf}_E f\geq r\} = \{E = \hat{X} : E\subseteq \{f\geq r\}\} = V_{\{f\geq r\}}\] is closed, proving that $\eta(f)$ is upper semicontinuous. It is clear that $\|\eta(f)\| = \|f\|$, so $\eta$ is an isometric embedding of $SC(X)$ into $SC(\hat{X})$. If $E$ is an irreducible closed subset of $X$, it is easy to see that $\eta(\chi_E) = \chi_{V_E}$, and thus $\eta$ sends $\mathrm{span}\{\chi_E : E\in \ms{E}(X)\}$ isomorphically onto $\mathrm{span}(\{\chi_{V_E} : E\in \ms{E}(X)\})$. These are dense in $SC(X)$ and $SC(\hat{X})$, respectively, by \hyperref[lemagree2]{Lemma~\ref*{lemagree2}}, so, extending by continuity, we obtain that $\eta$ is an isometric isomorphism $SC(X)\to SC(\hat{X})$.
\end{proof}

Propositions \ref{mono} and \ref{iso} combine to yield the following corollary, which gives the full description of the relationship between the measure theory of $X$ and that of $\hat{X}$.

\begin{cor} Let $X$ be a Noetherian space. Then the following diagram commutes. \smallskip

\[
\begin{psmatrix}[colsep= 1in, rowsep=.5 in]
\mc{M}(X) & \mc{M}(\hat{X})\\
SC(X)^* & SC(\hat{X})^*
\psset{arrows=->, nodesep=3pt}
\ncline{1,2}{2,2}
\trput{\Lambda}
\tlput{\cong}
\ncline{2,1}{2,2} 
\taput{(\eta^{-1})^*}
\tbput{\cong}
\psset{arrows=H->, hookwidth=-1.5mm, hooklength=2mm}
\ncline{1,1}{1,2}
\taput{j}
\ncline{1,1}{2,1}
\tlput{\Lambda}
\end{psmatrix}\] \smallskip

\noindent In particular, $\Lambda\colon \mc{M}(X)\to SC(X)^*$ is an isomorphism if and only if $X$ is complete.
\end{cor} 
\begin{proof} It suffices to check that $(\Lambda\circ j)(\delta_E) = ([\eta^{-1}]^*\circ\Lambda)(\delta_E)$ for all $E\in \ms{F}(X)$. Let $F\in \ms{E}(X)$ be an irreducible closed subset of $X$. Then \[
(\Lambda\circ j)(\delta_E)(\chi_{V_F}) = \Lambda(\delta_{E})(\chi_{V_F}) = \begin{cases} 1 & E\in  V_F\\ 0 & E\notin V_F\end{cases} = \begin{cases} 1 & E\subseteq F\\ 0  & E\not\subset F\end{cases} = \Lambda(\delta_E)(\chi_F).\] Since $\eta(\chi_F) = \chi_{V_F}$, this gives $(\Lambda\circ j)(\delta_E)(\chi_{V_F}) = [(\eta^{-1})^*\circ \Lambda](\delta_E)(\chi_{V_F})$. The corollary then follows by \hyperref[lemagree2]{Lemma~\ref*{lemagree2}}.
\end{proof}

\section{Invariant and ergodic measures}

In this section we begin our study of dynamics on Noetherian spaces. The dynamical systems we consider are continuous maps $f\colon X\to X$, where $X$ is a fixed nonempty Noetherian space. The goal of this section is to give a classification of all $f$-invariant and $f$-ergodic measures on $X$, proving \hyperref[thmD]{Theorem~\ref*{thmD}}. We begin by studying the push-forward operation $f_*$ on $\mc{M}(X)$.

\begin{prop}\label{ctn} The push-forward operator  $f_*\colon \mc{M}(X)\to \mc{M}(X)$ is continuous in both the weak and strong topologies.
\end{prop}
\begin{proof} Let $f^*\colon SC(X)\to SC(X)$ be the bounded linear operator given by $f^*\tau = \tau\circ f$. We begin by showing that $f_*$ is adjoint to $f^*$, or in other words, that $\Lambda(f_*\mu)(\tau) = \Lambda(\mu)(f^*\tau)$ for all $\mu\in \mc{M}(X)$ and $\tau\in SC(X)$. By \hyperref[lemagree2]{Lemma~\ref*{lemagree2}}, it suffices to prove this for $\tau = \chi_E$, where $E$ is a closed set. This is done easily: \[
\Lambda(f_*\mu)(\chi_E) = (f_*\mu)(E) = \mu(f^{-1}(E)) = \Lambda(\mu)(\chi_{f^{-1}(E)}) = \Lambda(\mu)(f^*\chi_E).\] We conclude that $\|\Lambda(f_*\mu)\|\leq \|\Lambda(\mu)\|\|f^*\|$, so $f_*$ is continuous in the strong topology. Now suppose $\tau\in SC(X)$ and $(\mu_\alpha)_{\alpha\in A}$ is a net in $\mc{M}(X)$ converging weakly to a measure $\mu$. Then  \[
\lim_\alpha \Lambda(f_*\mu_\alpha)(\tau) = \lim_\alpha \Lambda(\mu_\alpha)(f^*\tau) =  \Lambda(\mu)(f^*\tau) = \Lambda(f_*\mu)(\tau),\]  proving that $f_*$ is continuous in the weak topology.
\end{proof}

\begin{lem} Let $E\subseteq X$ be an irreducible (resp. $\sigma$-irreducible) closed set. Then $\ol{f(E)}$ is irreducible (resp. $\sigma$-irreducible). In particular, there is an induced map $\hat{f}\colon \hat{X}\to \hat{X}$ given by $\hat{f}(E) = \ol{f(E)}$. Moreover, the map $\hat{f}$ is continuous.
\end{lem}
\begin{proof} Assume that $E$ is irreducible. If $\ol{f(E)} = F_1\cup F_2$ for some closed sets $F_1$ and $F_2$, then $E = [E\cap f^{-1}(F_1)]\cup [E\cap f^{-1}(F_2)]$. Since $E$ is irreducible, this implies $E\cap f^{-1}(F_i) = E$ for some $i$, and hence  $f(E)\subseteq F_i$. Therefore $\ol{f(E)}\subseteq F_i$, from which it follows that $\ol{f(E)}$ is irreducible. A similar proof shows $\ol{f(E)}$ is $\sigma$-irreducible whenever $E$ is $\sigma$-irreducible. To see that $\hat{f}$ is continuous, it suffices to note that $\hat{f}^{-1}(V_F) = V_{f^{-1}(F)}$ for closed sets $F\subseteq X$.
\end{proof}

\begin{prop}\label{push} Let $\mu = \sum_{E\in \ms{F}} c_E\delta_E\in \mc{M}(X)$. Then $f_*\mu = \sum_{E\in \ms{F}} c_E\delta_{\ol{f(E)}}$. In particular, the following diagram commutes.

\[\begin{psmatrix}[colsep= 1in, rowsep=.5 in]
\mc{M}(X) & \mc{M}(X)\\
\mc{M}(\hat{X}) & \mc{M}(\hat{X})
\psset{arrows=->, nodesep=3 pt}
\ncline{1,1}{1,2}
\taput{f_*}
\ncline{2,1}{2,2}
\taput{\hat{f}_*}
\psset{arrows=H->, hookwidth=-1.5mm, hooklength=2mm}
\ncline{1,1}{2,1}
\tlput{j}
\ncline{1,2}{2,2}
\trput{j}
\end{psmatrix}\]
\end{prop}\smallskip

\begin{proof} We begin by showing that $f_*\delta_E = \delta_{\ol{f(E)}}$ for $E\in \ms{F}$. Let $F\subseteq X$ be a closed set. Then \[
(f_*\delta_E)(F) = \delta_E(f^{-1}(F)) = \begin{cases} 1 & E\subseteq f^{-1}(F)\\ 0 & \mbox{otherwise}\end{cases} = \begin{cases}1 & f(E)\subseteq F\\ 0 &\mbox{otherwise}\end{cases} = \begin{cases} 1 & \ol{f(E)}\subseteq F.\\ 0 & \mbox{otherwise}.\end{cases}\] Thus $f_*\delta_E$ agrees with $\delta_{\ol{f(E)}}$ on closed sets, so $f_*\delta_E = \delta_{\ol{f(E)}}$ by \hyperref[lemagree]{Lemma~\ref*{lemagree}}. Now assume that $\mu = \sum_{E\in \ms{F}} c_E\delta_E$ is an arbitrary measure on $X$. Let $E_1,E_2,\ldots$ be an enumeration of those $E\in \ms{F}$ for which $c_E\neq 0$. Since $\sum |c_E|<\infty$, the finite sums $\sum_{n=1}^Nc_{E_n}\delta_{E_n}$ converge strongly to $\mu$. By \hyperref[ctn]{Proposition~\ref*{ctn}}, one has \[
f_*\mu = f_*\lim_{N\to \infty} \sum_{n=1}^N c_{E_n}\delta_{E_n} = \lim_{N\to \infty} \sum_{n=1}^N c_{E_n}f_*\delta_{E_n} = \sum_{E\in \ms{F}} c_Ef_*\delta_E = \sum_{E\in \ms{F}} c_E\delta_{\ol{f(E)}}.\] This completes the proof.
\end{proof}

\hyperref[push]{Proposition~\ref*{push}} shows that, from the standpoint of measure theory, there is no loss of generality studying the dynamics of $\hat{f}$ instead of $f$; in fact, one actually gains information, as in general there will be more measures on $\hat{X}$ than on $X$. Thus for the remainder of the section and for much of the next, we will assume that $X$ is a Zariski space.

\begin{prop}\label{inv} Let $X$ be a Zariski space, and let $f\colon X\to X$ be a continuous map. Then a measure $\mu\in \mc{M}(X)$ is $f$-invariant if and only if it is of the form $\mu = \sum_{n=0}^\infty c_n\mu_n$, where the $c_n$ are absolutely summable real numbers and each $\mu_n$ is a probability measure of the form $\mu_n = r_n^{-1}(\delta_{x_1} + \delta_{x_2} + \cdots + \delta_{x_{r_n}})$ for some $f$-periodic cycle $x_1,\ldots, x_{r_n}\in X$.
\end{prop}
\begin{proof} Clearly any measure of the desired form $\mu = \sum c_n\mu_n$ is invariant, since each of the measures $\mu_n$ is invariant. Conversely, assume $\mu = \sum_{x\in X} a_x\delta_x$ is invariant. For each $\eps>0$, the set $S_\eps = \{x\in X : |a_x|>\eps\}$ is finite. Since $\mu$ is invariant, $f(S_\eps) = S_\eps$. Thus $S_\eps$ consists of finitely many periodic cycles for $f$. Moreover, the coefficients $a_x$ are constant along any such periodic cycle. The proposition follows immediately.
\end{proof}

\begin{proof}[Proof of Theorem \ref{thmD}] We recall that the $f$-ergodic probability measures are precisely those which are extremal in the convex set of all $f$-invariant probability measures (see Theorem 6.10.iii of \cite{MR648108}). Let $\mu$ be an $f$-invariant probability measure, and suppose $\mu = \sum c_n \mu_n$ is a decomposition as in \hyperref[inv]{Proposition~\ref*{inv}}. We may assume without loss of generality that $\mu_n\neq \mu_m$ for $n\neq m$ and that $c_n>0$ for all $n$. If there is more than one term in the sum, then \[
\mu = c_1\mu_1 + (1 - c_1)\sum_{n\geq 2} \frac{c_n}{1-c_1}\mu_n\] is a decomposition of $\mu$ as a convex combination of distinct $f$-invariant probability measures, so $\mu$ is not ergodic. Thus if $\mu$ is ergodic, it must be that $\mu = \mu_1$ is of the desired form. On the other hand, if $\mu = \mu_1$ is of the desired form, then it is clearly ergodic.
\end{proof}

\begin{rem} When $X$ is a Zariski space, then there is always an $f$-ergodic probability measure on $X$. Indeed, one can construct one explicitly as follows. Recursively define a sequence $X_n$ of closed sets by $X_0 = X$ and $X_{n+1} = \ol{f(X_n)}$ for each $n\geq 0$. The $X_n$ are a nested sequence of closed sets. Since $X$ is Noetherian, there is an $N$ such that $X_N = X_{N+1}$, or equivalently $X_N = \ol{f(X_N)}$. It follows that $f$ permutes the generic points of the irreducible components of $X_N$, so one obtains at least one periodic cycle for $f$ in this way.

On the other hand, if $X$ is not a Zariski space, it is possible that $X$ will have no ergodic probability measures. For example, if $X = \Z$ with the cofinite topology and $f\colon X\to X$ is the translation $f(n) = n+1$, then $X$ has no ergodic probability measures.
\end{rem}

\section{Asymptotic behavior of orbits}

Let $X$ be a nonempty Noetherian topological space, and $f\colon X\to X$ a continuous map. In this section we study the asymptotic behavior of both forward and reverse orbits of $f$ using the measure theory of $X$. As noted in the previous section, replacing $f$ with $\hat{f}$ if necessary allows us to assume that $X$ is a Zariski space. Throughout this section we will make this assumption, unless otherwise stated. We begin by studying the forward orbits of $f$.

\begin{Def} Let $x\in X$ be a point. The \emph{$\omega$-limit set} of $x$ is the closed set \[
L(x) := \bigcap_{k\geq 0}\ol{\{f^n(x) : n\geq k\}}\subseteq X.\]
\end{Def}

\begin{lem}\label{asymp} Suppose $L = L(x)$ is the $\omega$-limit set of a point $x\in X$. Then  \begin{enumerate}
\item[1.] The generic points of the irreducible components of $L$ form a periodic cycle for $f$.
\item[2.] There is an $N\geq 1$ such that $f^n(x)\in L$ for all $n\geq N$, and moreover $\{f^n(x): n\geq k\}$ is dense in $L$ for all $k\geq N$.
\end{enumerate} In particular, $L$ is the smallest closed set containing $f^n(x)$ for sufficiently large $n$.
\end{lem}
\begin{proof} By definition, $L$ is the intersection of a nested sequence of closed sets. Since $X$ is Noetherian, this sequence must stabilize. In other words, there is an index $N\geq 1$ such that $L = \ol{\{f^n(x) : n\geq k\}}$ for all $k\geq N$, proving (2). Suppose $L$ has irreducible decomposition $L =L_1\cup\cdots\cup L_r$, and let $y_i$ be the generic point of $L_i$ for each $i$. Observe that \[
\ol{f(L)} = \ol{f(\{f^n(x) : n\geq N\})} = \ol{\{f^{n}(x) : n\geq N+1\}} = L.\] It follows that $f$ permutes the $y_i$; to complete the proof, we must show that in fact $f$ acts transitively on the $y_i$. Up to relabeling, we may assume $f^N(x)\in L_1$, and that $\{y_1, y_2, \ldots, y_s\}$ is the orbit of $y_1$. Then $L = \ol{\{f^n(x) : n\geq N\}}\subseteq L_1\cup\cdots\cup L_s$, so one must have $r = s$. This completes the proof.
\end{proof} 

\begin{proof}[Proof of Theorem \ref{thmE}] Let $\mu_n = n^{-1}\sum_{k=0}^{n-1}f^k_*\delta_x$ for each $n\geq 1$. The space of Borel probability measures on $X$ is weakly sequentially compact by \hyperref[thmC]{Theorem~\ref*{thmC}}. Thus to prove the theorem, it suffices to show that every weakly convergent subsequence of $\mu_n$ converges  to the measure $r^{-1}(\delta_{y_1} + \cdots + \delta_{y_r})$, where the $y_i$ are the generic points of the irreducible components of the $\omega$-limit set $L = L(x)$. Fix a weakly convergent subsequence $\mu_{n_i}$ of $\mu_n$, say with $\mu_{n_i}\to \mu$ as $i\to \infty$. Since $f_*$ is weakly continuous by \hyperref[ctn]{Proposition~\ref*{ctn}},  we see that \[
f_*\mu = \lim_{i\to\infty} \frac{1}{n_i}\sum_{k=0}^{n_i-1} f^{k+1}_*\delta_x = \lim_{i\to \infty} \left(\mu_{n_i} + \frac{f^{n_i}_*\delta_x - \delta_x}{n_i}\right) = \lim_{i\to \infty} \mu_{n_i} = \mu.\] Thus $\mu$ is $f$-invariant, and we conclude by \hyperref[inv]{Proposition~\ref*{inv}} that $\mu = \sum_{j=1}^\infty c_j\nu_j$, where $c_j> 0$ for each $j$ and the $\nu_j$ are $f$-ergodic probability measures. Fix an index $j$, and suppose that $\nu_j = s^{-1}(\delta_{z_1} + \cdots + \delta_{z_s})$. Let $Z = \ol{\{z_1,\ldots, z_s\}}$. Since $z_1,\ldots, z_s$ is a periodic cycle for $f$, one has $f(Z) \subseteq Z$. Note that \[
0<c_j \leq \mu(Z) = \lim_{i\to \infty}\frac{1}{n_i}\sum_{k=0}^{n_i-1} f^k_*\delta_x(Z) = \lim_{i\to\infty} \frac{1}{n_i}\sum_{k=0}^{n_i-1}\chi_Z(f^k(x)).\] It follows that there is at least one $k$ such that $f^k(x)\in Z$, and hence that $f^n(x)\in Z$ for all $n\geq k$. As $L$ is the smallest closed set containing $f^n(x)$ for sufficiently large $n$, we conclude that $L\subseteq Z$. If $L\neq Z$, then $z_1,\ldots, z_s\notin L$, and thus $\mu(L) \leq 1 - c_j<1$. But \[
\mu(L) = \lim_{i\to \infty} \frac{1}{n_i}\sum_{k=0}^{n_i-1}\chi_L(f^k(x)) = 1\] since $f^k(x)\in L$ for all large enough $k$. Therefore $L = Z$, and $\nu_j = r^{-1}(\delta_{y_1} + \cdots + \delta_{y_r})$. As $j$ was an arbitrary index, in fact $\mu = r^{-1}(\delta_{y_1} + \cdots + \delta_{y_r})$. This completes the proof.
\end{proof}

While \hyperref[thmE]{Theorem~\ref*{thmE}} was originally proved by Favre (see \cite{thesis} and  \cite{MR1741274}), he stated the theorem in terms of semicontinuous functions instead of in terms of measures. We give this formulation in the following corollary.

\begin{cor}[Favre] \label{corfav} Let $X$ be an arbitrary Noetherian topological space, and $f\colon X\to X$ a continuous map. Fix a function $\tau\in SC(X)$ and a nonempty irreducible closed set $E\subseteq X$. Let $L\subseteq \hat{X}$ be the $\omega$-limit set of $E$ for $\hat{f}$, and let $y_1,\ldots, y_r$ be the generic points of the components of $L$. Then \[
\lim_{n\to \infty} \frac{1}{n}\sum_{k=0}^{n-1}\tau(\hat{f}^k(E)) = \frac{1}{r}[\tau(y_1) + \cdots + \tau(y_r)].\]
\end{cor}
\begin{proof} This is simply an application of \hyperref[thmE]{Theorem~\ref{thmE}} to the dynamical system $\hat{f}\colon \hat{X}\to \hat{X}$, written in terms of semicontinuous functions via the isomorphism $\mc{M}(\hat{X})\cong SC(\hat{X})^*$. 
\end{proof}

We now wish to do a similar analysis of the asymptotic behavior of the \emph{reverse} orbits of $f$, culminating in a proof of \hyperref[thmF]{Theorem~\ref*{thmF}}. For the rest of the section, we will need to assume that $f$ is surjective, as otherwise there may be points of $X$ which have no preimages at all, and hence no reverse orbit. Recall that we are assuming $X$ is a Zariski space.

\begin{Def} Let $x\in X$ be a point. A \emph{reverse orbit} of $x$ is a sequence $\{x_{-n}\}_{n=0}^\infty$ of points in $X$ such that $x_0 = x$ and $f(x_{-n}) = x_{-n+1}$ for all $n\geq 1$. If $f$ is surjective, every point has at least one reverse orbit.
\end{Def}

\begin{Def} Let $x\in X$ be a point, and let $\bk{x} = \{x_{-n}\}_{n=0}^\infty$ be a given reverse orbit of $x$. Define the \emph{$\alpha$-limit set} of  $\bk{x}$ to be the closed set \[
A(\bk{x}) := \bigcap_{k\geq 0} \ol{\{x_{-n} : n\geq k\}}\subseteq X.\] 
\end{Def}

\begin{lem}\label{alpha} Let $x\in X$, and let $\bk{x} = \{x_{-n}\}_{n=0}^\infty$ be a reverse orbit of $x$. Let $A$ be the $\alpha$-limit set of $\bk{x}$. Then  \begin{enumerate}
\item[1.] The generic points of the irreducible components of $A$ form a periodic cycle for $f$.
\item[2.] $A = \ol{\{x_{-n} : n\geq k\}}$ for every $k\geq 0$. 
\end{enumerate} In particular, $A = \ol{\{x_{-n} : n\geq 0\}}$.
\end{lem}
\begin{proof} By definition, $A$ is the intersection of a nested sequence of closed sets. Since $X$ is Noetherian, this sequence must stabilize. In other words, there is an index $N\geq 1$ such that $A = \ol{\{x_{-n} : n\geq k\}}$ for all $k\geq N$. Suppose $A$ has irreducible decomposition $A =A_1\cup\cdots\cup A_r$, and let $y_i\in X$ be the generic point of $A_i$ for each $i$. Observe that \[
\ol{f(A)} = \ol{f(\{x_{-n} : n\geq N+1\})} = \ol{\{x_{-n}: n\geq N\}} = A.\] It follows that $f$ permutes the $y_i$; to complete the proof, we must show that in fact $f$ acts transitively on the $y_i$. Since $x_{-n}\in A$ for $n$ sufficiently large, there must be at least one index $i$ such that $x_{-n}\in A_i$ for infinitely many $n$. Up to relabeling, we may assume without loss of generality that $i = 1$, and that $y_1,\ldots, y_s$ is the $f$-orbit of $y_1$.  It follows immediately that $x_{-n}\in A_1\cup\cdots\cup A_s$ for all $n$, and hence that $A\subseteq A_1\cup\cdots\cup A_s$. Therefore $r = s$, and one has $A = \{x_{-n} : n\geq k\}$ for every $k\geq 0$.
\end{proof}

\begin{proof}[Proof of Theorem \ref*{thmF}] Let $\mu_n = n^{-1}\sum_{k=0}^{n-1}\delta_{x_{-k}}$ for each $n\geq 1$. The space of Borel probability measures on $X$ is weakly sequentially compact by \hyperref[thmC]{Theorem~\ref*{thmC}}. Thus to prove the theorem, it suffices to show that every weakly convergent subsequence of $\mu_n$ converges to the measure $r^{-1}(\delta_{y_1} + \cdots + \delta_{y_r})$, where the $y_i$ are the generic points of the irreducible components of the $\alpha$-limit set $A = A(\bk{x})$. Fix a weakly convergent subsequence $\mu_{n_i}$ of $\mu_n$, say with $\mu_{n_i}\to \mu$ as $i\to \infty$. Since $f_*$ is weakly continuous by \hyperref[ctn]{Proposition~\ref*{ctn}}, we see that \[
f_*\mu = \lim_{i\to \infty} \frac{1}{n_i}\sum_{k=0}^{n_i-1}f_*\delta_{x_{-k}} = \lim_{i\to \infty}\left(\mu_{n_i} + \frac{f_*\delta_x - \delta_{x_{-n_i+1}}}{n_i}\right) = \lim_{i\to \infty} \mu_{n_i} = \mu.\] Thus $\mu$ is $f$-invariant, and we conclude by \hyperref[inv]{Proposition~\ref*{inv}} that $\mu = \sum_{j=1}^\infty c_j\nu_j$, where $c_j>0$ for each $j$ and the $\nu_j$ are $f$-ergodic probability measures. Fix an index $j$, and suppose that $\nu_j = s^{-1}(\delta_{z_1} + \cdots + \delta_{z_s})$. Let $Z = \ol{\{z_1,\ldots, z_s\}}$. Since $z_1,\ldots, z_s$ is a periodic cycle for $f$, one has $f(Z)\subseteq Z$. Note that \[
0<c_j \leq \mu(Z) = \lim_{i\to \infty} \frac{1}{n_i}\sum_{k=0}^{n_i-1} \delta_{x_{-k}}(Z) = \lim_{i\to \infty} \frac{1}{n_i}\sum_{k=0}^{n_i-1}\chi_Z(x_{-k}).\] There must then be infinitely many $k$ such that $x_{-k}\in Z$, and hence $x_{-k}\in Z$ for all $k$, since $f(Z)\subseteq Z$. Because $A$ is the smallest closed set containing each of the $x_{-k}$, we conclude that $A\subseteq Z$. If $A\neq Z$, then  $z_1,\ldots, z_s\notin A$, thus $\mu(A)\leq 1 - c_j<1$. But \[
\mu(A) = \lim_{i\to\infty} \frac{1}{n_i}\sum_{k=0}^{n_i-1}\chi_A(x_{-k})= 1\] since $x_{-k}\in A$ for all $k$. Therefore $A = Z$, and $\nu_j = r^{-1}(\delta_{y_1} + \cdots + \delta_{y_r})$. As $j$ was an arbitrary index, it follows that $\mu = r^{-1}(\delta_{y_1}  +\cdots + \delta_{y_r})$. This completes the proof.
\end{proof}

\begin{cor} Let $X$ be an arbitrary Noetherian topological space, and let $f\colon X\to X$ be a surjective continuous map. Let $\tau\in SC(X)$. Fix a nonempty irreducible closed set $E\subseteq X$ as well as a $\hat{f}$-reverse orbit $\bk{E} = (E_{-n})_{n=0}^\infty$ for $E$. Let $A\subseteq \hat{X}$ be the $\alpha$-limit set of $\bk{E}$, and let $y_1,\ldots, y_r$ be the generic points of the irreducible components of $A$. Then \[
\lim_{n\to \infty} \frac{1}{n}\sum_{k=0}^{n-1} \tau(E_{-k}) = \frac{1}{r}[\tau(y_1) + \cdots + \tau(y_r)].\]
\end{cor}
\begin{proof} This is simply an application of \hyperref[thmF]{Theorem~\ref*{thmF}} to the dynamical system $\hat{f}\colon \hat{X}\to \hat{X}$, written in terms of semicontinuous functions via the isomorphism $\mc{M}(\hat{X})\cong SC(\hat{X})^*$.
\end{proof}

In \cite{MR2513537}, Dinh gives a different analysis of the asymptotic behavior of reverse orbits  in the specific case that $f$ is an open endomorphism of a compact complex analytic space. For the remainder of the section, we will expand on his results, proving them for more general Noetherian spaces and relating them to \hyperref[thmE]{Theorem~\ref*{thmE}} and \hyperref[thmF]{Theorem~\ref*{thmF}}. We once again assume $f\colon X\to X$ is a surjective continuous map, with $X$ a nonempty Zariski space.

Fix a bounded upper semicontinuous function $\tau$ on $X$. For each $n\geq 1$,  define \[
\tau_n := \sum_{k=0}^{n-1} \tau\circ f^k.\] Note that the $\tau_n$ are again bounded upper semicontinuous functions on $X$. For any $x\in X$,  \hyperref[thmE]{Theorem~\ref*{thmE}} asserts that  \[
\tau_+(x):= \lim_{n\to \infty} \frac{\tau_n(x)}{n} = \frac{1}{r}(\tau(y_1) + \cdots + \tau(y_r)),\] where the $y_i$ are the generic points of the irreducible components of the $\omega$-limit set of $E$. To study reverse orbits, Dinh defines functions $\tau_{-n}\colon X\to \R$ by \[
\tau_{-n}(x) := \sup_{f^n(y) = x} \tau_n(y),\] for each $n\geq 1$, and considers the analogous quantity \begin{equation}\label{eqn}
\tau_{-}(x):= \lim_{n\to \infty} \frac{\tau_{-n}(x)}{n}.\end{equation} We spend the rest of the section proving the following theorem, which generalizes the results in \cite{MR2513537}.

\begin{thm}\label{dinh} The limit in \hyperref[eqn]{Equation~(\ref*{eqn})} exists for any point $x\in X$, and the limit function $\tau_-\colon X\to \R$ satisfies the following properties: \begin{enumerate}
\item[1.] For any $x\in X$, one has $\tau_-(x) = \max \tau_+(y)$, where the maximum is taken over periodic points $y$ such that $x\in \ol{\{y\}}$.
\item[2.] For any $x\in X$, one has $\tau_-(x) = \max \lim_{n\to \infty}\tau_n(x_{-n})/n$, where the maximum is taken over all reverse orbits $\{x_{-n}\}_{n=0}^\infty$ of $x$.
\item[3.] The function $\tau_-\colon X\to \R$ is upper semicontinuous.
\item[4.] One has $\tau_-(x)\leq \tau_+(x)$ for all $x\in X$. If $x$ is periodic, then equality holds.
\end{enumerate}
\end{thm}

The bulk of the proof will be showing that the limit in \hyperref[eqn]{Equation~(\ref*{eqn})} exists; the remaining statements will follow without too much work from the proof of this fact. We should note that the argument given below is similar in spirit to the argument given by Dinh, but differs in some significant ways. For instance, Dinh heavily uses the fact that his spaces $X$ are of finite Krull dimension and that his maps $f$ preserve Krull dimension. We make no such assumptions.

\begin{lem}\label{jump} Let $\tau\colon X\to \R$ be an upper semicontinuous function on $X$, and let $c\in \R$. Then there is an $a<c$ such that $\tau(x)<a$ for any $x\in X$ with $\tau(x)<c$.
\end{lem}
\begin{proof} For any real value $a\in \R$, the set $X_a = \{x\in X : \tau(x)\geq a\}$ is closed, and moreover $X_a\subseteq X_{b}$ whenever $b\leq a$. Thus the family $X_a$ of sets with $a<c$ is a nested family of closed sets. Since $X$ is Noetherian, there must be an $a<c$ such that $X_b = X_{a}$ for all $a<b<c$. This completes the proof.
\end{proof}

\begin{prop}\label{unif} Fix $c\in \R$, and let $Z = \{x\in X: \tau_n(x)\geq cn\mbox{ for all }n\geq 1\}$. Then there is a real number $b<c$ and an integer $N\geq 1$ with the following property: if $n\geq N$ and $x\in X$ is such that $f^k(x)\notin Z$ for all $k = 0,\ldots, n$, then $\tau_n(x)\leq bn$.
\end{prop}
\begin{proof} Let $V_n = \{x\in X : \tau_n(x)\geq cn\}$ for each $n\geq 1$, and let $U_n = X\smallsetminus V_n$. Note that the $V_n$ are closed, since the $\tau_n$ are upper semicontinuous. By definition, $Z = \bigcap V_n$. Since $X$ is Noetherian, there is an integer $M\geq 1$ such that $Z = V_1\cap \cdots\cap V_M$. Using \hyperref[jump]{Lemma~\ref*{jump}}, there is a real number $a<c$ such that $\tau_n(x)<an$ for any $x\in U_n$, where $n = 1,\ldots, M$. We choose $N$ large enough that $a + \|\tau\|M/N<c$, and let $b = a + \|\tau\|M/N$.

Suppose that $n\geq N$ and that $x\in X$ is such that $f^k(x)\notin Z$ for all $k = 0,\ldots, n$. We recursively define a finite sequence $k_i$ of integers as follows. First, we set $k_1 = 0$. Assume now that $k_i$ has been defined. If $n - k_i\leq M$, we stop defining the $k_i$. If $n - k_i>M$, then by hypothesis $f^{k_i}\in U_j$ for some $j = 1,\ldots, M$. We then set $k_{i+1} = k_i + j$. Let $k_1,\cdots, k_\ell$ be the sequence constructed in this fashion. By construction, one has $0\leq n - k_\ell\leq M$ and $\tau_{k_\ell}(x)\leq ak_\ell$. It follows that \[
\tau_n(x) = \tau_{k_\ell}(x) + \tau_{n - k_\ell}(f^{k_\ell}(x))\leq ak_\ell + \|\tau\|(n - k_\ell)\leq an + \|\tau\|M = (a + \|\tau\|M/n)n\leq bn,\] as desired. 
\end{proof}

\begin{prop}\label{obvious} Suppose $y\in X$ is periodic and $x\in \ol{\{y\}}$. Then \[
\liminf_{n\to \infty} \tau_{-n}(x)/n\geq \tau_+(y).\]
\end{prop}
\begin{proof} Let $\{y_0,\ldots, y_{r-1}\}$ be the orbit of $y$, where without loss of generality $y = y_0$ and $f(y_i) = y_{i-1}$, the indices taken modulo $r$. One may then choose a reverse orbit $\{x_{-n}\}_{n=0}^\infty$ of $x$ such that $x_{-n}\in \ol{\{y_n\}}$, for all $n$, where the indices for the $y_n$ are taken modulo $r$. Then \[
\frac{\tau_{-n}(x)}{n} \geq \frac{\tau_n(x_{-n})}{n}\geq \frac{\tau_n(y_n)}{n}\] for all $n$. One easily sees that  $\tau_n(y_n)/n\to \tau_+(y)$ as $n\to \infty$.
\end{proof}

\begin{prop}\label{exists} For each $x\in X$, the limit $\tau_-(x) := \lim_{n\to \infty}\tau_{-n}(x)/n$ exists.
\end{prop}
\begin{proof} Let $c = \limsup_{n\to \infty} \tau_{-n}(x)/n$. Let $Z = \{y\in X : \tau_n(y)\geq cn\mbox{ for all }n\geq 1\}$. Choose $b$ and $N$ as in \hyperref[unif]{Proposition~\ref*{unif}}.  Let $Z$ have irreducible decomposition $Z_1\cup\cdots\cup Z_r$, and let $z_i$ be the generic point of $Z_i$ for each $i$. Let $L(z_i)$ be the $\omega$-limit set of $z_i$. We begin by showing that $x\in L(z_i)$ for some $i$. Suppose for contradiction that $x\notin L(z_1)\cup\cdots\cup L(z_r)$. By \hyperref[asymp]{Lemma~\ref*{asymp}}, there is an integer $s\geq 0$ such that $f^s(Z_i)\subseteq L(z_i)$ for each $i$. Thus if $n\geq s$ and $y\in X$ is such that $f^n(y) = x$, it follows that $y\notin Z$. If $n\geq s+N$, \hyperref[unif]{Proposition~\ref*{unif}} then implies \[
\frac{\tau_{-n}(x)}{n} \leq\frac{1}{n}(b(n-s) + s\|\tau\|)\to b,\] a contradiction of $\limsup_{n\to \infty} \tau_{-n}(x)/n = c>b$. Therefore $x\in L(z_i)$ for some $i$, without loss of generality for $i = 1$. Let $F$ be a component of $L(z_1)$ such that $x\in F$, and let $y$ be its generic point. Then $y$ is periodic, and  \hyperref[thmE]{Theorem~\ref*{thmE}} gives $\tau_+(y) = \tau_+(z_1)$. By the definition of $Z$, one has $\tau_+(z_1)\geq c$. We conclude by \hyperref[obvious]{Proposition~\ref*{obvious}} that $\liminf_{n\to \infty}\tau_{-n}(x)/n \geq c$.
\end{proof}

\begin{proof}[Proof of Theorem \ref{dinh}] Now that we have proved the existence of the limits $\tau_-(x)$, we can easily prove statements 1 through 4. Fix $x\in X$.

(1) Let $c = \tau_-(x)$. We saw in the proof of \hyperref[exists]{Proposition~\ref*{exists}} that $x$ is contained in an irreducible closed set $F$ with periodic generic point $y$ such that $\tau_+(y) \geq c$. On the other hand, if $x$ is contained in an irreducible closed set $F$ with periodic generic point $y$, then \hyperref[obvious]{Proposition~\ref*{obvious}} implies that $\tau_+(y)\leq c$. This completes the proof of (1).

(2) One clearly has $\tau_-(x)\geq \lim_{n\to \infty}\tau_n(x_{-n})/n$ for any given reverse orbit $\{x_{-n}\}_{n=0}^\infty$ of $x$. By statement (1), there is an irreducible closed set $F$ with periodic generic point $y$ containing $x$ such that $\tau_+(y) = \tau_-(x) = c$. Let $y_0,\ldots, y_{r-1}$ be the orbit of $y$, where $y=y_0$ and $f(y_i) = y_{i-1}$ for each $i$, the indices being taken modulo $r$. We may then choose a reverse orbit $\{x_{-n}\}_{n=0}^\infty$ of $x$ such that $x_{-n}\in \ol{\{y_n\}}$ for each $n$, where the indices of $y_n$ are taken modulo $r$.  Then $c = \tau_+(y)\leq \lim_{n\to \infty}\tau(x_{-n})/n\leq \tau_-(x) = c$. This completes (2).

(3) Let $c\in \R$, and let $Z = \{y\in X : \tau_n(y)\geq cn\mbox{ for all }n\geq 1\}$. Let $Z = Z_1\cup\cdots\cup Z_r$ be the irreducible decomposition of $Z$, and let $z_i$ be the generic point of $Z_i$ for each $i$. We saw in the proof of \hyperref[exists]{Proposition~\ref*{exists}} that $\tau_-(x)\geq c$ if and only if $x\in L(z_i)$ for some $i$. Thus $\tau_-$ is upper semicontinuous.

(4) Let $F$ be an irreducible closed set containing $x$ with periodic generic point $y$ such that $\tau_-(x) = \tau_+(y)$. Then $L(y)\supseteq L(x)$, so $\tau_+(y)\leq \tau_+(x)$. Suppose, on the other hand, that $x$ is itself periodic, say with orbit $x_0,\ldots, x_{r-1}$, where without loss of generality $x_0 = x$ and $\hat{f}(x_i) = x_{i-1}$, the indices taken modulo $r$. Then for the specific reverse orbit $\{x_{-n}\}_{n=0}^\infty$ of $x$, one sees that $\tau_+(x) =\lim_{n\to \infty}\tau_n(x_{-n})/n \leq \tau_-(x)$. Thus $\tau_+(x) = \tau_-(x)$.
\end{proof}

\bibliographystyle{alpha}
%\addcontentsline{toc}{chapter}{References}
\bibliography{References}

\end{document}